\documentclass[11pt]{article}
\usepackage{amsthm, amsmath, amssymb, amsfonts, tikz, setspace, fancyhdr, bm}
\usepackage[margin = 2.2cm]{geometry}
\usepackage{hyperref, enumerate}

\usepackage[english]{babel}
\usepackage[babel]{microtype}
\usepackage[capitalise]{cleveref}
\usepackage{comment}
\usepackage{graphicx}
\usepackage{float}
\usepackage{mathtools}
\usepackage{xcolor}

\theoremstyle{definition}
\newtheorem{theorem}{Theorem}[section]

\newtheorem{prop}[theorem]{Proposition}
\newtheorem{lemma}[theorem]{Lemma}
\newtheorem{fact}[theorem]{Fact}

\newtheorem{conjecture}[theorem]{Conjecture}

\newtheorem{claim}[theorem]{Claim}

\newenvironment{poc}{\begin{proof}[Proof of the claim]}{\end{proof}}

\newcommand*{\eqdef}{\stackrel{\mbox{\normalfont\tiny def}}{=}} % definition by equality

\newcommand*{\maximal}{critical}
\newcommand{\ex}{\mathrm{ex}}

\newcommand{\cH}{\mathcal{H}}

\newcommand{\cT}{\mathcal{T}}

\newcommand{\boxp}{\mathbin{\square}}

\author{
	Zichao Dong\thanks{Extremal Combinatorics and Probability Group (ECOPRO), Institute for Basic Science (IBS), Daejeon, South Korea. Supported by the Institute for Basic Science (IBS-R029-C4). \texttt{$\{$zichao,jungao,hongliu$\}$@ibs.re.kr}. }
	\and Jun Gao\footnotemark[1]
	\and Hong Liu\footnotemark[1]
}
\date{}

\title{Bipartite Tur\'an problems via graph gluing}

\begin{document}
	
	\maketitle
	\begin{abstract}
		For graphs $H_1$ and $H_2$, if we glue them by identifying a given pair of vertices $u \in V(H_1)$ and $v \in V(H_2)$, what is the extremal number of the resulting graph $H_1^u \odot H_2^v$? In this paper, we study this problem and show that interestingly it is equivalent to an old question of Erd\H{o}s and Simonovits on the Zarankiewicz problem. When $H_1, H_2$ are copies of a same bipartite graph $H$ and $u, v$ come from a same part, we prove that $\operatorname{ex}(n, H_1^u \odot H_2^v) = \Theta \bigl( \operatorname{ex}(n, H) \bigr)$. As a corollary, we provide a short self-contained disproof of a conjecture of Erd\H{o}s, which was recently disproved by Janzer. 
	\end{abstract}
	
	
	\section{Introduction} \label{sec:intro}
	
	The extremal number of a graph $H$, denoted by $\ex(n, H)$, is the maximum number of edges in an $n$-vertex $H$-free graph. Determining this function is one of the most important topics in extremal graph theory, originating from the classical work of Mantel~\cite{1907Mantel} and Tur\'an~\cite{1941Turan}. For any graph $H$ and $n \to \infty$, the celebrated Erd\H{o}s--Stone--Simonovits theorem~\cite{1966ESS,1946ErdosBAMS} states that 
	\[
	\ex(n, H) = \Bigl( 1- \frac{1}{\chi(H)-1} + o(1) \Bigr) \cdot \frac{n^2}{2}, 
	\]
	where $\chi(H)$ is the chromatic number of $H$. This result asymptotically resolves the problem for every non-bipartite $H$. However, for a bipartite $H$, it only shows $\ex(n, H) = o(n^2)$. Complete comprehension of the behavior of $\ex(n,H)$ for bipartite $H$ remains elusive. The order of magnitude is known only for a handful of bipartite graphs. For more details, we refer the readers to the F\"{u}redi--Simonovits survey~\cite{2013FSsurvey}. 
	
	As every graph can be built from smaller ones, a natural approach is to study how graph operations could affect the extremal function for bipartite graphs. In this paper, we investigate the following simple gluing operation: given bipartite $H_1, H_2$ and vertices $u \in V(H_1), \, v \in V(H_2)$, denote by $H_1^u \odot H_2^v$ the graph built from gluing $H_1$ and $H_2$ by identifying $u$ and $v$. Throughout this paper, the big-O, little-o, and big-Theta notations always hide the $n \to \infty$ limit process. When a gluing $H_1^u \odot H_2^v$ is considered, we implicitly assume that the vertex sets $V(H_1), V(H_2)$ are disjoint. %That is, $V(H_1) \cap V(H_2) = \varnothing$. 
	We propose the following conjecture. 
	
	\begin{conjecture} \label{conj: mian}
		If $H_1, H_2$ are bipartite graphs and $u \in V(H_1), \, v \in V(H_2)$, then
		\[
		\ex(n, H_1^u \odot H_2^v) = \Theta\bigl( \ex(n, H_1) + \ex(n, H_2) \bigr). 
		\]
	\end{conjecture}
	
	The following special case of~\Cref{conj: mian} when $H_1$ and $H_2$ are isomorphic is particularly interesting, as we shall see that it is closely related to the famous Zarankiewicz problem. 
	
	\begin{conjecture}\label{conj: mian2}
		If $H_1, H_2$ are two copies of a bipartite graph $H$ and $u \in V(H_1), \, v \in V(H_2)$, then
		\[
		\ex(n, H_1^u \odot H_2^v) = \Theta\bigl( \ex(n, H) \bigr). 
		\]
	\end{conjecture}
	
	\newpage
	
	To explain the connection to the Zarankiewicz problem, define $\ex(n, m, H)$ as the maximum number of edges of a subgraph of the complete bipartite graph $K_{n, m}$ that does not contain $H$ as a subgraph. The Zarankiewicz problem is an asymmetric version: given a bipartition $(L, R)$ of a bipartite graph $H$, its Zarankiewicz number, denoted by $z(n, m, H[L, R])$, is the maximum number of edges of a subgraph of $K_{n, m}$ that does not contain $H$ as a subgraph with $L$ in the part of size $n$ and $R$ in the part of size $m$. When $n = m$, we write $z(n, n, H) \eqdef z(n, n, H[L, R]) = z(n, n, H[R, L])$ for ease of notations.
	
	It follows directly from the definitions that $z(n, n, H) \ge \ex(n, n, H)$ and $\ex(2n, H) \ge \ex(n, n, H)$. Since every graph contains a balanced bipartite subgraph with at least half of its edges, we have that $\ex (n,n,H) \ge \frac{1}{2} \ex(2n,H)$. For the relationship between $\ex(n, n, H)$ and $z(n, n, H)$, Erd\H os and Simonovits~\cite{1984Simonovits} made the following conjecture (see \cite[Conjecture 2.12]{2013FSsurvey}).
	
	\begin{conjecture}[\cite{1984Simonovits}] \label{conj:erdos}
		If $H$ is a bipartite graph, then
		\[
		z(n, n, H) = \Theta\bigl( \ex(n, n, H) \bigr). 
		\]
	\end{conjecture}
	
	Since the extremal number of the union of two disjoint graphs $H_1$ and $H_2$ has the same order as $\max \bigl\{ \ex(n, H_1), \ex(n, H_2) \bigr\}$, it suffices to consider connected graphs in these conjectures. 
	
	Let $H_1$ and $H_2$ be two copies of a connected bipartite graph $H$ with bipartition $(A,B)$. For any $u \in A$ and $v \in B$, we have $z(n,n, H) \le \ex(n,n,H_1^u\odot H_2^v)$. So, \Cref{conj: mian2} implies \Cref{conj:erdos}. 
	
	Our first result shows that these three conjectures are in fact equivalent to each other.
	
	\begin{theorem} \label{thm: main 1}
		\Cref{conj: mian}, \Cref{conj: mian2}, and \Cref{conj:erdos} are all equivalent. 
	\end{theorem}
	
	Considering \Cref{conj: mian2}, our next result states that if the two vertices we merge come from the same part of the graph, then the conjecture holds. 
	
	\begin{theorem} \label{thm: main 2}
		Let $H_1$ and $H_2$ be two copies of a connected bipartite  graph $H$ with bipartition $(A, B)$. If vertices $u \in V(H_1), \, v \in V(H_2)$ satisfy that $u, v$ are both from $A$ or both from $B$, then
		\[
		\ex(n, H_1^u\odot H_2^v) = \Theta\bigl( \ex(n, H) \bigr). 
		\]
	\end{theorem}
	
	\Cref{thm: main 2} is related to yet another old conjecture of Erd\H{o}s and Simonovits. For any positive integer $r$, recall that a graph is \emph{$r$-degenerate} if each of its subgraphs has minimum degree at most $r$. Erd\H{o}s and Simonovits~\cite{erdos1988} proposed the conjecture below. 
	
	\begin{conjecture}[\cite{1981ESconj}] \label{conj:r=2}
		For a bipartite $H$, we have $\ex(n, H) = O(n^{3/2})$ if and only if $H$ is $2$-degenerate.    
	\end{conjecture}
	
	Erd\H{o}s was particularly interested in \Cref{conj:r=2} as he stated the problem several times \cite{1981ESconj,erdos1983,erdos1988,erdos1993}. Moreover, he offered {\$}250 for a proof and {\$}500 for a disproof. \Cref{conj:r=2} has the following extension. 
	
	\begin{conjecture} \label{conj:r}
		For a bipartite $H$, we have $\ex(n, H) = O(n^{2-1/r})$ if and only if $H$ is $r$-degenerate. 
	\end{conjecture} 
	
	Historically, the ``if'' direction of \Cref{conj:r} was made by Erd\H{o}s \cite{erdos1967}. As for the other direction, a stronger conjecture that any bipartite $H$ of minimum degree at least $r + 1$ has $\ex(n, H) = \Omega(n^{2-1/r+\varepsilon})$ was proposed by Erd\H{o}s and Simonovits \cite{erdos1993}. The ``only if'' direction was recently disproved by Janzer, first for every $r \ge 3$ in~\cite{janzer_cycle}, and then for $r=2$ in~\cite{janzer23}, via delicate and involved explicit constructions. 
	
	As an application of \Cref{thm: main 2}, we present here a very short disproof of \Cref{conj:r} for all $r \ge 2$ at once by arguing that the ``if'' and the ``only if'' direction cannot hold simultaneously. It is worth mentioning that Janzer's result is significantly stronger, as he constructed, for every $\varepsilon > 0$, a $3$-regular bipartite graph $H$ with $\ex(n, H) = O(n^{4/3 + \varepsilon})$. Call a graph $G$ \emph{\maximal} $r$-degenerate if $G$ is $r$-degenerate in which exactly one vertex is of degree $r$ and every other vertex is of degree at least $r + 1$. 
	
	\begin{proof}[Disproof of Conjecture~\ref{conj:r} assuming \Cref{thm: main 2}]
		We first prove that for every $r \ge 2$, there exists a bipartite {\maximal} $r$-degenerate graph. Construct in three steps: 
		\vspace{-0.5em}
		\begin{itemize}
			\item To start, we take a copy of $K_{r, r+1}$ whose bipartition is $(X, Y)$ such that $|X| = r$ and $|Y| = r+1$. 
			\vspace{-0.5em}
			\item For all but one $r$-element subsets $S \subseteq Y$, we add a unique vertex $z_S$ whose neighborhood is $S$. 
			\vspace{-0.5em}
			\item Finally, we add yet another new vertex $w$ which is adjacent to exactly all those vertices $z_S$. 
		\end{itemize}
		\vspace{-0.5em}
		We obtain a critical $r$-degenerate graph. See Figure~\hyperlink{figone}{1} for the $r = 2$ case. 
		
		\vspace{0.5em}
		\centerline{
			\begin{tikzpicture}[scale=1]
				\clip (-6, -2.5) rectangle (9, 1.5);
				\node[shape = circle,draw = black,fill,inner sep=0pt,minimum size=1.0mm] at (0,0.5) {};
				\node[shape = circle,draw = black,fill,inner sep=0pt,minimum size=1.0mm] at (0,-0.5) {};
				\node[shape = circle,draw = black,fill,inner sep=0pt,minimum size=1.0mm] at (1,-1) {};
				\node[shape = circle,draw = black,fill,inner sep=0pt,minimum size=1.0mm] at (1,0) {};
				\node[shape = circle,draw = black,fill,inner sep=0pt,minimum size=1.0mm] at (1,1) {};
				\node[shape = circle,draw = black,fill,inner sep=0pt,minimum size=1.0mm] at (2,0.5) {};
				\node[shape = circle,draw = black,fill,inner sep=0pt,minimum size=1.0mm] at (2,-0.5) {};
				\node[shape = circle,draw = black,fill,inner sep=0pt,minimum size=1.0mm] at (3,0) {};
				\draw[thick, fill opacity=0.3] 
				(0,0.5) -- (1,-1)-- (0,0.5) -- (1,1)-- (0,0.5) -- (1,0);
				\draw[thick, fill opacity=0.3] 
				(0,-0.5) -- (1,-1)-- (0,-0.5) -- (1,1)-- (0,-0.5) -- (1,0);
				\draw[thick, fill opacity=0.3] 
				(1,1) -- (2,0.5)-- (1,0) -- (2,-0.5)-- (1,-1);
				\draw[thick, fill opacity=0.3] 
				(2,0.5)-- (3,0) -- (2,-0.5);
				\node at (3.25, 0) {$w$};
				\node at (2.5, 0.65) {$z_{\{y_1, y_2\}}$};
				\node at (2.47, -0.78) {$z_{\{y_2, y_3\}}$};
				\node at (1, 1.24) {$y_1$};
				\node at (1, 0.25) {$y_2$};
				\node at (1, -1.22) {$y_3$};
				\node at (-0.25, 0.5) {$x_1$};
				\node at (-0.27, -0.5) {$x_2$};
				\node at (1.5, -2) {\textbf{\hypertarget{figone}{Figure 1}:} A {\maximal} $2$-degenerate graph with $w$ being the unique degree-$2$ vertex. };
			\end{tikzpicture}
		}
		
		Pick a bipartite {\maximal} $r$-degenerate $H$. Assuming the ``if'' direction of \Cref{conj:r=2}, we have $\ex(n, H) = O(n^{2-1/r})$. Take two copies $H_1, H_2$ of $H$ and denote by $u, v$ the unique degree-$r$ vertices of $H_1, H_2$, respectively. Write $H^* \eqdef H_1^u \odot H_2^v$. Then $H^*$ is not $r$-degenerate with $\ex(n, H^*) = O(n^{2-1/r})$ (by \Cref{thm: main 2}), contradicting the ``only if'' direction of \Cref{conj:r=2}. 
	\end{proof}
	
	The ``if'' direction of~\Cref{conj:r=2} remains a major open problem, even for the case $r = 2$: whether every $2$-degenerate bipartite $H$ satisfies $\ex(n, H) = O(n^{3/2})$. All known $2$-degenerate graphs with such properties have at least two vertices of degree at most $2$ (e.g.,~grids \cite{2023BJS} and blow-ups of trees \cite{2022GJNL}). This makes us wonder the following question: does there exist a {\maximal} $2$-degenerate graph $H$ such that $\ex(n, H) = O(n^{3/2})$. Moreover, if such graphs do exist, then we can avoid using the ``if'' direction of~\Cref{conj:r=2} in the disproof above. In our next result, we construct infinitely many such graphs, answering this question in the affirmative. 
	
	\begin{theorem}\label{thm: main 3}
		There are infinitely many {\maximal} $2$-degenerate graphs $H$ such that $\ex(n, H) = O(n^{3/2})$. 
	\end{theorem}
	
	The smallest graph we constructed in~\Cref{thm: main 3} (Figure~\hyperlink{figone}{2}) has $54$ vertices. It would be interesting to find smaller examples. In particular, does the graph in Figure~\hyperlink{figone}{1} have extremal number $O(n^{3/2})$?
	
	\paragraph{Paper organization.} We deduce a useful lemma concerning the graph gluing operation, and prove \Cref{thm: main 1,thm: main 2} in \Cref{sec:glue}. We derive \Cref{thm: main 3} via explicit constructions in \Cref{sec:example}. 
	
	\section{Proofs of \texorpdfstring{\Cref{thm: main 1}}{Theorem 1} and \texorpdfstring{\Cref{thm: main 2}}{Theorem 2}} \label{sec:glue}
	
	In this section, we begin with stating the key lemma concerning the behavior of extremal numbers under the vertex-gluing operation. From this lemma we can quickly derive \Cref{thm: main 1,thm: main 2}. The proof of the key lemma is given at the end.
	
	For any graph $G$, we denote by $v(G)$ and $e(G)$ the numbers of its vertices and edges, respectively, and $\delta(G), \Delta(G)$ its minimum and maximum vertex degree, respectively. For a vertex $v$ in $G$, denote by $N(v)$ and $\deg(v)$ the neighborhood and the degree of $v$, respectively. We have $|N(v)| = \deg(v)$. 
	
	\begin{lemma} \label{lem:key}
		For any $C, \alpha > 0$, there exists some sufficiently large $N = N_{C, \alpha} > 0$ with the following property: Let $H_1, H_2$ be bipartite graphs with vertices $u \in V(H_1), \, v \in V(H_2)$ and $G$ be a bipartite graph with bipartition $(L, R)$ such that $n \eqdef v(G) \ge N, \, e(G) \ge Cn^{1+\alpha}$. If for every $L' \subseteq L$ and $R' \subseteq R$ with $e(G[L', R']) \ge \frac{e(G)}{48v(H_1)}$, the induced subgraph $G[L', R']$ contains both a copy of $H_1$ and a copy of $H_2$ in which $u, v$ come from $R'$, then $G$ contains a copy of $H_1^u \odot H_2^v$. 
	\end{lemma}
	
	Before proving \Cref{thm: main 1,thm: main 2}, we recall a folklore result in graph theory. 
	
	\begin{fact} \label{fact: bipartite subgraph}
		Every graph $G$ contains a balanced bipartite subgraph on at least $e(G)/2$ edges. 
	\end{fact}
	
	\begin{proof}[Proof of \Cref{thm: main 1} assuming \Cref{lem:key}]
		We have already seen 
		\[
		\text{``\Cref{conj: mian} $\implies$ \Cref{conj: mian2} $\implies$ \Cref{conj:erdos}''}
		\]
		in \Cref{sec:intro}. It suffices to show ``\Cref{conj:erdos} $\implies$ \Cref{conj: mian}''. 
		
		Write $m_i = \ex(n, H_i)$ for $i = 1, 2$ and assume without loss of generality that $m_1 \ge m_2$. We need to prove $\ex(n, H_1^u \odot H_2^v) = \Theta(m_1)$. The lower bound is straightforward, since every $H_1$-free graph is $H_1^u \odot H_2^v$-free. To see the upper bound, we may assume that $H_1$ contains an even cycle, because $\ex(n, H)$ is linear in $n$ if and only if the bipartite graph $H$ is a forest. This implies that $m_1 \ge Kn^{1+\alpha}$ for some $K, \alpha > 0$. For convenience, we assume $n$ is even. 
		
		Due to the assumption that \Cref{conj:erdos} holds, there exists a constant $C > 0$ such that for $i \in [2]$,
		\[
		z(n/2, n/2, H_i) \le C \cdot \ex(n/2, n/2, H_i) \le C \cdot \ex(n, H_i) = Cm_i\le Cm_1. 
		\]
		Let $G$ be a graph with $e(G) > 96Cv(H_1)m_1$. We are to show that $G$ contains a copy of $H_1^u \odot H_2^v$. 
		
		Thanks to \Cref{fact: bipartite subgraph}, we can find a bipartite subgraph $G'$ of $G$ with partition $(L, R)$ such that $|L| = |R| = n/2$ and $e(G)> 48Cv(H_1)m_1$. For every $L' \subseteq L$ and $R' \subseteq R$ with 
		\[
		e(G[L', R']) \ge \frac{e(G)}{48v(H_1)} > Cm_1  \ge \max\bigl\{z(n/2, n/2, H_1), z(n/2, n/2, H_2)\bigr\}, 
		\]
		the induced subgraph $G[L', R']$ contains a copy of $H_1$ and a copy of $H_2$, where both $u$ and $v$ are from $R'$. From \Cref{lem:key} we deduce that $G'$ (hence $G$) contains a copy of $H_1^u \odot H_2^v$. 
	\end{proof}
	
	\begin{proof}[Proof of \Cref{thm: main 2} assuming \Cref{lem:key}]
		Since $H$ is a subgraph of $H_1^u \odot H_2^v$, it suffices to establish $\ex(n, H_1^u \odot H_2^v) = O\bigl( \ex(n, H) \bigr)$. This is trivial if $H$ is acyclic, for the extremal number of every tree is linear in $n$. We then assume that $H$ contains an even cycle, hence $\ex(n, H) \ge C n^{1+\alpha}$ for some $C, \alpha > 0$. 
		
		Set $m \eqdef v(H)$. Let $G$ be an $n$-vertex graph. Thanks to \Cref{fact: bipartite subgraph}, we may assume further that $G$ is a bipartite graph with bipartition $(L, R)$ and $n \ge N_{C, \alpha}$ (as in \Cref{lem:key}), $e(G) > (48m)^2 \cdot \ex(n, H)$. We are to show that $G$ contains a copy of $H^* \eqdef H_1^u \odot H_2^v$. Assume without loss of generality that $u, v$ both come from $A$ in the bipartition $(A, B)$ of $H$. The proof idea can be informally explained as follows. If every large subgraph $G[L', R']$ of $G[L, R]$ contains $H$ with $u, v \in R'$, then \Cref{lem:key} tells us that we can find $H^*$ by embedding $(A, B)$ into $(R, L)$. Otherwise, there exists a large subgraph $G[L', R']$ containing no $H$ with $u, v \in R'$. Since $H$ is connected, this implies that every large subgraph $G[L'', R'']$ of $G[L', R']$ contains $H$ with $u, v \in L''$, and hence we can find $H^*$ by embedding $(A, B)$ into $(L', R')$. 
		
		Formally, if the induced subgraph $G[L', R']$ contains a copy of $H$ with $u, v \in R'$ for every pair of $L' \subseteq L, \, R' \subseteq R$ with $e(G[L', R']) \ge \frac{e(G)}{48m}$, then \Cref{lem:key} implies that $G$ contains a copy of $H^*$. Otherwise, there exist subsets $L' \subseteq L, \, R' \subseteq R$ with
		\[
		e(G[L', R']) \ge \frac{e(G)}{48m} > (48m) \cdot \ex(n, H) \ge Cn^{1+\alpha}
		\]
		such that $G' \eqdef G[L', R']$ does not contain any copy of $H$ with $u, v \in R'$. For any $L'' \subseteq L'$ and $R'' \subseteq R'$ with $e(G'[L'', R'']) \ge \frac{e(G')}{48m} > \ex(n, H)$, since $G'[L'', R'']$ contains a copy of $H$ while $G'$ does not contain any copy of $H$ with $u, v \in R'$, the fact that $u, v \in A$ and $H$ is connected implies that $G'[L'', R'']$ contains a copy of $H$ with $u, v \in L''$. Again, it follows from \Cref{lem:key} that $G'$ contains a copy of $H^*$. 
	\end{proof}
	
	Before proving \Cref{lem:key}, we recall the Chernoff bound and a folklore graph theory result: 
	
	\begin{prop}[{\cite[Corollary 21.7]{frieze_karonski}}] \label{chernoff}
		Let $X_1, \dots, X_n$ be $\{0, 1\}$-valued independent random variables. Write $X \eqdef \sum_{i=1}^n X_i$ and $\mu \eqdef \mathbb{E}(X)$. Then $\mathbb{P} \bigl[ |X-\mu| \ge \delta\mu \bigr] \le 2e^{-\delta^2\mu/3}$ holds for every $\delta \in [0, 1]$. 
	\end{prop}
	
	\begin{fact} \label{fact: min-degree}
		Every $n$-vertex graph $G$ contains a subgraph $H$ with $e(H) \ge e(G)/2$ and $\delta(H) \ge e(G)/(2n)$. 
	\end{fact}
	
	Let $G$ be a bipartite graph with bipartition $(L, R)$. For any $\varepsilon > 0$, call $(L_1, L_2)$ an $\varepsilon$-\emph{good partition} of $L$ if $|N(v) \cap L_i| = (1/2 \pm \varepsilon) \deg(v)$ holds for all $v \in R$ and $i = 1, 2$. That is, 
	\[
	\biggl| |N(v) \cap L_i| - \frac{\deg(v)}{2} \biggr| \le \varepsilon \deg(v). 
	\]
	
	\begin{lemma}\label{lem: good partition}
		For any $\varepsilon, C, \alpha > 0$, there exists some sufficiently large $N = N_{\varepsilon, C, \alpha} > 0$ with the following property: Let $G$ be a bipartite graph with bipartition $(L, R)$ such that $n \eqdef v(G) \ge N$ and $\delta(G) \ge Cn^{\alpha}$. Then there exists an $\varepsilon$-good partition of $L$. 
	\end{lemma}
	
	\begin{proof}
		Let $L_1, L_2$ be a uniform random partition of $L$. For each $v \in R$, we write $X_v \eqdef |N(v) \cap L_1|$, and so $\mathbb{E}(X_v) = \deg(v)/2$. It then follows from the Chernoff bound (\Cref{chernoff}) that
		\[
		\mathbb{P} \bigl[ |X_v-\deg(v)/2| \ge \varepsilon \deg(v) \bigr] \le 2e^{-2\varepsilon^2\deg(v)/3} \le 2e^{-2\varepsilon^2\delta(G)/3} \le 2e^{-2\varepsilon^2 C n^\alpha/3}. 
		\]
		So, when $n$ is sufficiently large in terms of $\varepsilon, C, \alpha$, from the union bound we deduce that 
		\[
		\mathbb{P} \bigl[ \text{$|X_v-\deg(v)/2| < \varepsilon \deg(v)$ holds for each $v \in R$} \bigr] \ge 1 - n \cdot 2e^{-2\varepsilon^2 C n^\alpha/3} > 0. 
		\]
		Thus, with positive probability $(L_1, L_2)$ gives an $\varepsilon$-good partition of $L$, as desired. 
	\end{proof}
	
	\begin{proof}[Proof of \Cref{lem:key}]
		Due to~\Cref{fact: min-degree}, there exists a subgraph $G^*$ of $G$ with $\delta(G^*) \ge e(G)/(2n) \ge Cn^\alpha/2$ and $e(G^*) \ge e(G)/2$. Let $(L^*, R^*)$ be a bipartition of $G^*$ with $L^* \subseteq L, \, R^* \subseteq R$. By~\Cref{lem: good partition}, there is a $\frac{1}{4}$-good partition $(L_1, L_2)$ of $L^*$. For $i = 1, 2$, set $G_i \eqdef G^*[L_i, R^*]$. Then $\frac{1}{4} e(G^*) \le e(G_i) \le \frac{3}{4} e(G^*)$. 
		
		In $G_1$, we are going to take $F_1, \dots, F_k$, a sequence of copies of $H_1$ such that in each $F_i$, the vertex corresponding to $u$ lies in $R^*$. Let $u_i$ be the copy of $u$ in $F_i$ and write $S_i \eqdef V(F_i) \cap R^*$ for $i = 1, \dots, k$ with $S_0 \eqdef \varnothing$. For $i = 0, 1, \dots$, conduct the following algorithm: 
		
		\vspace{-0.5em}
		\begin{itemize}
			\item Suppose $F_1, \dots, F_i$ have been constructed. For $j \le i$, set $S'_j \eqdef \bigl\{ x \in S_j : \deg_{G_1}(x) \le \deg_{G_1}(u_j) \bigr\}$. 
			\vspace{-0.5em}
			\begin{itemize}
				\item If there exists a copy of $H_1$ in $G_1$ whose vertex set is disjoint from $\bigcup_{j=1}^i S_j'$, then define $F_{i+1}$ as an arbitrary such copy maximizing $\deg_{G_1}(u_{i+1})$. 
				\item Otherwise, the process halts with $k \eqdef i$ and $X \eqdef \{u_1, \dots, u_k\}$. 
			\end{itemize}
		\end{itemize}
		
		\begin{claim} \label{cl:X}
			We have $X \cap S_i = \{u_i\}$ for $i = 1, \dots, k$. We also have $e(G^*[L_1, X]) \ge \frac{e(G_1)}{2v(H_1)}$. 
		\end{claim}
		
		\begin{poc}
			Assume to the contrary that there exist distinct $i, j \in \{1, \dots, k\}$ satisfying $u_i \in S_j$. If $i < j$, then $u_i \in S'_i \cap S_j = \emptyset$, a contradiction. If $i > j$, then $u_i \in S_j \setminus S'_j$, and so $\deg_{G_1}(u_i) > \deg_{G_1}(u_j)$, contradicting the maximum assumption on $\deg_{G_1}(u_j)$. We thus conclude the first part of the claim. 
			
			To see the second part, denote $S' \eqdef \bigcup_{i=1}^{k} S'_i$. There is no copy of $H_1$ in $G_1$ with $u \in R^* \setminus S'$. So, 
			\[
			e(G_1[L_1, R^* \setminus S']) < \frac{e(G_1)}{48v(H_1)} < \frac{1}{2} e(G_1) \implies e(G_1[L_1, S']) \ge \frac{1}{2} e(G_1). 
			\]
			Observe that $\sum_{x \in S_i'} \deg_{G_1}(x) \le |S_i'| \cdot \deg_{G_1}(u_i) \le |S_i| \cdot \deg_{G_1}(u_i) \le v(H_1) \cdot \deg_{G_1}(u_i)$. We thus obtain
			\[
			v(H_1) \cdot e(G^*[L_1, X]) = v(H_1) \sum_{i=1}^k \deg_{G_1}(u_i) \ge \sum_{i=1}^k \sum_{x \in S'_i} \deg_{G_1}(x) = e(G_1[L_1, S']) \ge \frac{1}{2} e(G_1), 
			\]
			which concludes the second part of the claim. 
		\end{poc}  
		
		Since $(L_1, L_2)$ is a $\frac{1}{4}$-good partition, from \Cref{cl:X} we deduce that
		\begin{align*}
			e(G^*[L_2, X]) \ge \frac{1}{3} e(G^*[L_1, X]) \ge \frac{e(G_1)}{6v(H_1)} \ge \frac{e(G^*)}{24v(H_1)} \ge \frac{e(G)}{48v(H_1)}. 
		\end{align*}
		So, our assumption on $G$ implies that $G_2$ contains a copy of $H_2$ (denoted by $H_2'$) such that the vertex $v$ of $H_2$ (denoted by $v'$) in $H_2'$ appears in $X$. Let $t$ be the index such that $v' = u_t$. Then \Cref{cl:X} shows that $V(H'_2) \cap V(F_t) = \{v'\} = \{u_t\}$, and hence $F_t$ together with $H_2'$ gives a copy of $H_1^u \odot H_2^v$. 
	\end{proof}

	\section{Critical \texorpdfstring{$2$}{2}-degenerate graphs} \label{sec:example}
	
	For graphs $G$ and $H$, the \emph{Cartesian product} $G \boxp H$ is the graph on vertex set $V(G) \times V(H)$, where two vertices $(u, v)$ and $(u', v')$ are adjacent if and only if either $u = u'$ and $\{v, v'\} \in E(H)$, or $v = v'$ and $\{u, u'\} \in E(G)$. An edge between $(u, v)$ and $(u', v')$ is of \emph{type} $H$ if $u = u'$, and of \emph{type} $G$ if $v = v'$. 
	
	For every integer $\ell \ge 3$, we define the \emph{prism graph} $C^{\boxp}_{\ell} \eqdef C_\ell \boxp K_2$. This graph $C^{\boxp}_{\ell}$ consists of two disjoint $\ell$-cycles whose edges are of type $C_{\ell}$ and an $\ell$-matching whose edges are of type $K_2$. Recently, Gao, Janzer, Liu, and Xu~\cite{gao_extremal_2023} established $\ex(n, C^{\boxp}_{2\ell}) = \Theta_{\ell}(n^{3/2})$ for all $\ell \ge 4$. 
	
	Let $C^{\Join}_{2\ell}$ be the graph obtained by gluing two copies of $C^{\boxp}_{2\ell}$ along one edge of type $K_2$, and $C^{\Join-}_{2\ell}$ be the graph obtained from $C^{\Join}_{2\ell}$ by removing the edge $e_1$ being merged and another edge $e_2$ sharing a single vertex with $e_1$. See Figure~\hyperlink{figtwo}{2} for an illustration. 
	
	\centerline{
		\begin{tikzpicture}
			\clip (-3,-3.25) rectangle (14,2.5);			
			\draw[thick,red,dashed] (5.5,0.5) -- (7,1.75);
			\draw[thick,red,dashed] (5.5,0.5) -- (5.5,-0.5);
			\node at (5.72,0) {\textcolor{red}{$e_1$}};
			\node at (6.05,1.25) {\textcolor{red}{$e_2$}};		
			\foreach \m in {-1,0,1,2,3,4,7,8,9,10,11,12}
			{
				\foreach \n in {0.75, 1.75, -0.75, -1.75}
				\draw[fill = black] (\m,\n) circle (0.075);
				\draw[thick] (\m,0.75) -- (\m,1.75);
				\draw[thick] (\m,-0.75) -- (\m,-1.75);
			}
			\foreach \m in {1,2,3,4,5,6}
			{
				\node at (\m+6,2){\textcolor{blue}{\footnotesize $\m$}};
			}
			\node at (13.5,0.75){\textcolor{blue}{\footnotesize $7$}};		
			\foreach \m in {8,9,10,11,12,13}
			{
				\node at (20-\m,-0.5){\textcolor{blue}{\footnotesize $\m$}};
			}
			\node at (5.5,0.8){\textcolor{blue}{\footnotesize $14$}};		
			\foreach \m in {15,16,17,18,19,20}
			{
				\node at (19-\m,2){\textcolor{blue}{\footnotesize $\m$}};
			}
			\node at (-2.5,0.75){\textcolor{blue}{\footnotesize $21$}};	
			\foreach \m in {22,23,24,25,26,27}
			{
				\node at (\m-23,-0.5){\textcolor{blue}{\footnotesize $\m$}};
			}
			\foreach \m in {28,29,30,31,32,33}
			{
				\node at (32-\m,-2){\textcolor{blue}{\footnotesize $\m$}};
			}
			\node at (-2.5,-0.75){\textcolor{blue}{\footnotesize $34$}};
			\foreach \m in {35,36,37,38,39,40}
			{
				\node at (\m-36,0.5){\textcolor{blue}{\footnotesize $\m$}};
			}
			\node at (5.5,-0.8){\textcolor{blue}{\footnotesize $41$}};
			\foreach \m in {42,43,44,45,46,47}
			{
				\node at (\m-35,-2){\textcolor{blue}{\footnotesize $\m$}};
			}
			\node at (13.5,-0.75){\textcolor{blue}{\footnotesize $48$}};
			\foreach \m in {49,50,51,52,53,54}
			{
				\node at (61-\m,0.5){\textcolor{blue}{\footnotesize $\m$}};
			}
			\draw[fill = black] (-2.5,0.5) circle (0.075);
			\draw[fill = black] (-2.5,-0.5) circle (0.075);
			\draw[fill = black] (5.5,0.5) circle (0.075);
			\draw[fill = black] (5.5,-0.5) circle (0.075);
			\draw[fill = black] (5.5,0.5) circle (0.075);
			\draw[fill = black] (5.5,-0.5) circle (0.075);
			\draw[fill = black] (13.5,0.5) circle (0.075);
			\draw[fill = black] (13.5,-0.5) circle (0.075);
			\draw[thick] (-2.5,0.5) -- (-2.5,-0.5);
			\draw[thick] (-2.5,0.5) -- (-1,1.75);
			\draw[thick] (-2.5,0.5) -- (-1,-0.75);
			\draw[thick] (-2.5,-0.5) -- (-1,-1.75);
			\draw[thick] (-2.5,-0.5) -- (-1,0.75);
			\draw[thick] (13.5,0.5) -- (13.5,-0.5);
			\draw[thick] (13.5,0.5) -- (12,1.75);
			\draw[thick] (13.5,0.5) -- (12,-0.75);
			\draw[thick] (13.5,-0.5) -- (12,-1.75);
			\draw[thick] (13.5,-0.5) -- (12,0.75);
			\draw[thick] (5.5,0.5) -- (7,-0.75);
			\draw[thick] (5.5,-0.5) -- (7,-1.75);
			\draw[thick] (5.5,-0.5) -- (7,0.75);
			\draw[thick] (5.5,0.5) -- (4,1.75);
			\draw[thick] (5.5,0.5) -- (4,-0.75);
			\draw[thick] (5.5,-0.5) -- (4,-1.75);
			\draw[thick] (5.5,-0.5) -- (4,0.75);
			\draw[thick] (-1,0.75) -- (4,0.75);
			\draw[thick] (-1,1.75) -- (4,1.75);
			\draw[thick] (-1,-0.75) -- (4,-0.75);
			\draw[thick] (-1,-1.75) -- (4,-1.75);
			\draw[thick] (7,0.75) -- (12,0.75);
			\draw[thick] (7,1.75) -- (12,1.75);
			\draw[thick] (7,-0.75) -- (12,-0.75);
			\draw[thick] (7,-1.75) -- (12,-1.75);
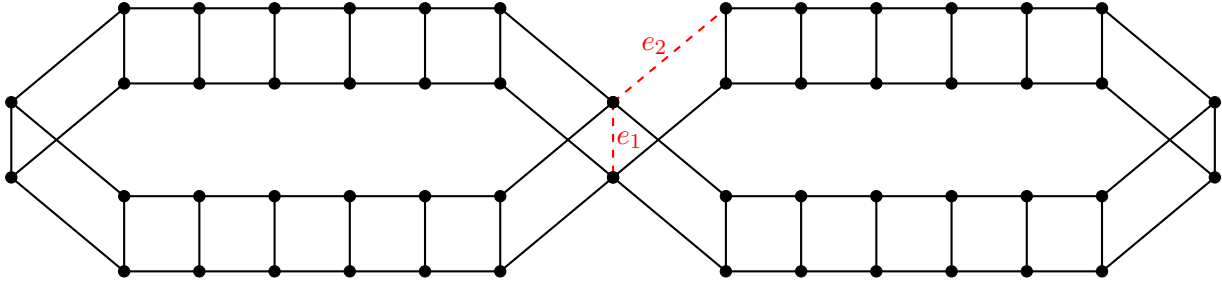
			\node at (5.5,-2.75) {\textbf{\hypertarget{figtwo}{Figure 2}:} The graph $C^{\Join-}_{14}$ with a vertex ordering.}; 
		\end{tikzpicture}
	}
	
	In Figure~\hyperlink{figtwo}{2}, the blue ordering implies that $C^{\Join-}_{2\ell}$ is a {\maximal} $2$-degenerate graph. We shall prove $\ex(C^{\Join-}_{2\ell}) = O_{\ell}(n^{3/2})$ for each $\ell \ge 7$, which offers an infinite family of graphs promised by~\Cref{thm: main 3}. 
	
	\begin{theorem}\label{thm: 2-de}
		For any integer $\ell \ge 7$, we have $\ex(n, C^{\Join-}_{2\ell}) \le \ex(n, C^{\Join}_{2\ell}) = O(n^{3/2})$. 
	\end{theorem}
	
	\subsection{Proof of \texorpdfstring{\Cref{thm: 2-de}}{Theorem 3.1}}
	
	To begin with, we clean up our host graph by finding a well-behaved almost regular subgraph. For $K > 0$, a graph $G$ is $K$-\emph{almost regular} if $\Delta(G) \le K\delta(G)$. This kind of regularization was developed by Erd\H{o}s and Simonovits \cite{1970Erdos}. We are going to use the following variant. 
	
	\begin{lemma}[{\cite[Proposition 2.7]{2012SIAMJiang}}]\label{lem:JiangSeiverK-almost}
		Suppose $n \gg C \ge 1$. If $G$ is an $n$-vertex graph with $e(G) \ge Cn^{3/2}$, then $G$ has a $10^3$-almost regular subgraph $G'$ on $m \ge n^{1/12}$ vertices such that $e(G') \ge (C/3) \cdot m^{3/2}$. 
	\end{lemma}
	
	By $n \gg C$ we refer to that $n$ is sufficiently large in terms of $C$. Write $K \eqdef 10^3, \, T \eqdef 10^3K\ell = 10^6\ell$ and consider parameters satisfying $n \gg C \gg T > \ell \ge 7$. Let $G$ be an $n$-vertex graph of average degree $d \eqdef 2e(G)/n = Cn^{1/2}$. (Notice that we can remove extra egdes.) We are going to show that $G$ contains a copy of $C_{2\ell}^{\Join}$. Due to \Cref{fact: bipartite subgraph} and \Cref{lem:JiangSeiverK-almost}, we may assume that $G$ is bipartite $K$-almost regular. The following supersaturation result shows that we can find many $4$-cycles to work with. 
	
	\begin{lemma}[{\cite[Theorem 6]{1983ES}}] \label{Supersaturation of cycles}
		Suppose $C \ge 10$. If $G$ is an $n$-vertex graph with $Cn^{3/2}$ edges, then $G$ contains at least $C^4 n^2/2$ copies of $4$-cycles. 
	\end{lemma}
	
	%When we refer to $x_1 x_2 \cdots x_{\ell}$ as an $\ell$-cycle, we implicitly assume that the edges are $x_1x_2, \dots, x_{\ell}x_1$. For $u, v \in V(G)$, denote by $\deg(u, v)$ their \emph{codegree}, which is the number of vertices that are adjacent to both $u$ and $v$. Call a $4$-cycle $xyzw$ \emph{thin} if the codegrees of both diagonal pairs $\deg(x, z), \, \deg(y, w)$ are upper bounded by $Td^{1/2}$, and \emph{thick} otherwise. \textcolor{red}{It follows from $C \gg K$ and
		%\begin{align*}
		%    \hom(C_4, G) \ge \sum_{x, y \in V(G)} |N(x, y)|^2 \ge n^{-2} \biggl( \sum_{x, y \in V(G)} |N(x, y)| \biggr)^2 \ge n^{-2} \bigl( \hom(K_{1, 2}, G) \bigr)^2
		%\end{align*}
		%that $G$ contains $cd^4$ many $4$-cycles, where $c \eqdef 2^{-5}$. So, there are many thick or many thin $4$-cycles.} 
	
	When we refer to $x_1 x_2 \cdots x_{\ell}$ as an $\ell$-cycle, we implicitly assume that the edges are $x_1x_2, \dots, x_{\ell}x_1$. For $u, v \in V(G)$, denote by $\deg(u, v)$ their \emph{codegree}, which is the number of vertices that are adjacent to both $u$ and $v$. Call a $4$-cycle $xyzw$ \emph{thin} if the codegrees of both diagonal pairs $\deg(x, z), \, \deg(y, w)$ are upper bounded by $Td^{1/2}$, and \emph{thick} otherwise. From \Cref{Supersaturation of cycles} we deduce that $G$ contains at least $cd^4$ copies of $4$-cycles, where $c \eqdef 2^{-5}$. So, there are either many thick $4$-cycles or many thin $4$-cycles in $G$. 
	
	\medskip
	
	\noindent\textbf{Case 1.} The number of thick $4$-cycles in $G$ is at least $c d^4/2$. 
	
	\medskip
	
	Without loss of generality, we assume that there are at least $cd^{4}/4$ many thick $4$-cycles $xyzw$ in $G$ satisfying $\deg(y, w) \ge Td^{1/2}$. The pigeonhole principle then shows that there is an edge $xy \in E(G)$ that is contained in $(cd^4/4)/e(G) \ge cC^2d/2 > Td$ such thick $4$-cycles. This means the number of $4$-cycles $xyzw$ with $\deg(y, w) \ge Td^{1/2}$ is at least $Td$. (Here $x, y$ are fixed vertices while $z, w$ vary.) 
	
	\centerline{
		\begin{tikzpicture}
			\clip (-1, -1.25) rectangle (9, 1.5);
			\draw[thick] (0, 0) grid (8, 1);
			\foreach \m in {0,1,2,3,4,5,6,7,8}
			\foreach \n in {0,1}
			\draw[fill = black] (\m,\n) circle (0.075);
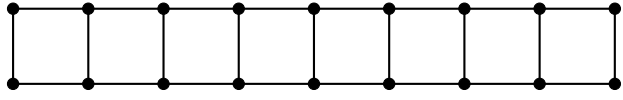
			\node at (4, -0.75) {\textbf{\hypertarget{figthree}{Figure 3}:} The graph $P^{\boxp}_8$.}; 
		\end{tikzpicture}
	}
	
	Denote by $P_t^{\boxp} \eqdef P_t \boxp K_2$ the Cartesian product of the $t$-edge path $P_t$ and the single edge $K_2$ (see Figure~\hyperlink{figthree}{3} for an illustration). With the help of the lemma below, we are going to find a $P_{4\ell}^{\boxp}$ in $G$ first, and then finish the proof by ``building up'' a $C_{2\ell}^{\Join}$ from it. 
	
	\begin{lemma}[{\cite[Lemma 2.6]{gao_extremal_2023}}] \label{lem:p2timespt}
		Let $H$ be a bipartite graph with bipartition $(X, Y)$. If $e(H) \ge 20t|Y|$ and $\deg(x) \ge 20t|Y|^{1/2}$ holds for every $x \in X$, then $H$ contains a copy of $P_t^{\boxp}$. 
	\end{lemma}
	
	Take $X \eqdef \{w \in N(x) \setminus \{y\} : \deg(y, w) \ge Td^{1/2}\}, \, Y \eqdef N(y) \setminus \{x\}$ and let $H \eqdef G[X, Y]$. By definition, every thick $4$-cycle containing $xy$ corresponds to an edge in $H$. So, the $K$-almost regularity of $G$ implies that $e(H) > Td > (80K\ell)d \ge 80\ell|Y|$. For each $w \in X$, since $\deg_H(w) \ge \deg(y, w) - 1 \ge 80\ell|Y|^{1/2}$, by \Cref{lem:p2timespt} with $t \eqdef 4\ell$ we deduce that $H$ has a copy of $P_{4\ell}^{\boxp}$, which contains two vertex-disjoint copies of $P_{2\ell-2}^{\boxp}$. By including the additional vertices $x$ and $y$, we obtain a copy of $C^{\Join}_{2\ell}$, as desired. 
	
	\medskip
	\noindent\textbf{Case 2.} The number of thin $4$-cycles in $G$ is at least $c d^4/2$. 
	\medskip
	
	We need a technical result inspired by \cite{janzer23}. Let $\cH = (V, E)$ be an $n$-vertex graph and suppose $\to$ is a binary relation (not necessarily symmetric) defined over $V$. If $|\{w \in N(v) : u \to w\}| \le \beta\deg(v)$ holds for any pair of (not necessarily distinct) vertices $u, v \in V$, then we say that $\to$ is a $\beta$-\emph{nice} binary relation on $\cH$. The following lemma is an asymmetric version of \cite[Lemma 2.5]{janzer23}. 
	
	\begin{lemma} \label{lem:goodC2k}
		Let $\ell \ge 2$ be an integer. Suppose $\cH = (V, E)$ is an $n$-vertex non-empty graph and $\to$ is a $\beta$-nice binary relation on $\cH$. If $\beta < \bigl(10^7 \ell^3 n^{1/\ell} (\log n)^4 \bigr)^{-1}$, then $\cH$ contains a homomorphic $2\ell$-cycle $x_1x_2 \cdots x_{2\ell}$ such that the relation $x_i \to x_j$ fails for all distinct pairs of indices $i, j$. 
	\end{lemma}
	
	\Cref{lem:goodC2k} follows from almost identical arguments in the proof of \cite[Lemma 2.5]{janzer23}. To make the exposition self-contained, we shall include a sketch highlighting the differences after proving \Cref{thm: 2-de}. 
	
	\medskip
	
	Construct an auxiliary graph $\Gamma$ with $V(\Gamma) = E(G)$. For every pair of distinct edges $e_1 = xy$ and $e_2 = zw$, put $e_1e_2$ into $E(\Gamma)$ if $xyzw$ or $xywz$ forms a thin $4$-cycle in $G$. Then $e(\Gamma) \ge cd^4/2$. The idea is to find a $C_{2\ell} \odot C_{2\ell}$ in $\Gamma$, whose counterpart in the original graph $G$ contains a $C_{2\ell}^{\Join}$ subgraph. 
	
	By \Cref{fact: bipartite subgraph} and \Cref{fact: min-degree}, there is a bipartite subgraph $\cH$ of $\Gamma$ with $m \eqdef v(\cH) \le e(G) = dn/2$ and
	\[
	e(\cH) \ge e(\Gamma)/4 \ge cd^4/8, \qquad \delta(\cH) \ge e(\Gamma)/(2m) \ge cd^4/(4m). 
	\]
	From $d \ge Cn^{1/2}$ we deduce that $d = \Omega_{\ell}(m^{1/3})$, and so $e(\cH) = \Omega_{\ell}(m^{4/3}), \, \delta(\cH) = \Omega_{\ell}(m^{1/3}) = \Omega_{\ell}(d)$. 
	
	For every pair of distinct $e_1, e_2 \in V(\cH)$, we write $e_1 \sim e_2$ if and only if they share a common vertex. Observe that $\sim$ is a symmetric binary relation over $V(\cH)$.
	
	\begin{claim}\label{claim: beta good}
		The binary relation $\sim$ is $\beta$-nice on $\cH$ for some $\beta = O_{\ell}(m^{-1/6})$.
	\end{claim}
	
	\begin{poc}
		For any $e_1, e_2 \in V(\cH)$, the definition of thin $4$-cycles suggests that $e_2$ has at most $4Td^{1/2}$ neighbors $e_3 \in V(\cH)$ with $e_1 \sim e_3$. Indeed, each such neighbor $e_3$ is determined by picking one vertex $a \in e_2$, one vertex $b$ of $e_1$, and finding a thin $4$-cycle with $a, b$ being a pair of diagonal vertices. It follows from $4Td^{1/2}/\deg_{\cH}(e_1) \le 4Td^{1/2}/\delta(\cH) = O_{\ell}(d^{-1/2}) = O_{\ell}(m^{-1/6})$ that $\sim$ is $O_{\ell}(m^{-1/6})$-nice. 
	\end{poc}
	
	Call a homomorphic $2\ell$-cycle $x_1 x_2 \cdots x_{2\ell}$ in $\cH$ \emph{good} if $x_i \nsim x_j$ for any $i \ne j$, and \emph{bad} otherwise. 
	
	\begin{claim} \label{claim: goodcycle}
		Any subgraph $\cH'$ of $\cH$ with $e(\cH') \ge e(\cH)/2$ contains a good homomorphic $2\ell$-cycle. 
	\end{claim}
	
	\begin{poc}
		Recall that $v(\cH) = m \le dn/2$. \Cref{fact: min-degree} shows that $\cH'$ contains a subgraph $\cH''$ with 
		\[
		\delta(\cH'') \ge e(\cH')/(2m) \ge e(\cH)/(4m) = \Omega_{\ell}(d). 
		\]
		By the proof of \Cref{claim: beta good}, $\sim$ is $\beta''$-nice on $\cH''$ for some $\beta'' = O_{\ell}(\beta)$. Since $m \ge \delta(\cH) \to \infty$, we have
		\[
		\beta'' = O_{\ell}(\beta) \overset{\text{\Cref{claim: beta good}}}{=} O_{\ell}(m^{-1/6}) \le \Omega_{\ell} \bigl( m^{-1/\ell} (\log m)^{-4} \bigr) = \Bigl(10^7 \ell^3 v(\cH'')^{1/\ell} \bigl( \log v(\cH'') \bigr)^4 \Bigr)^{-1}
		\]
		because $\ell \ge 7$. Thus, \Cref{lem:goodC2k} shows that $\cH''$ (hence $\cH'$) contains a good homomorphic $2\ell$-cycle. 
	\end{poc}
	
	Recall that $\cH$ is a bipartite graph. Let $(L, R)$ be a bipartition of $\cH$. By \Cref{claim: goodcycle}, we can greedily pick good homomorphic $2\ell$-cycles $C_1, \dots, C_t$ in $\cH$ with $L \cap V(C_i) \cap V(C_j) = \varnothing$ for any distinct indices $i, j$, until $e \bigl( \cH \bigl[ L \setminus \bigcup_{i=1}^t V(C_i), R \bigr] \bigr) < e(\cH)/2$. (This infers that $e \bigl( \cH \bigl[ L \setminus \bigcup_{i=1}^{t-1} V(C_i), R \bigr] \bigr) \ge e(\cH)/2$.) The definition of $\sim$ tells us that each $C_i$ corresponds to a copy of $C_{2\ell}^{\boxp}$ in $G$. 
	
	Denote $S \eqdef L \cap \bigl( \bigcup_{i=1}^t V(C_i) \bigr)$ and ${\cH}''' \eqdef \mathcal{H}[S, R]$. Then $e({\cH}''') \ge e(\cH)/2$. 
	\vspace{-0.5em}
	\begin{itemize}
		\item For each $e \in S$, there is a unique $C_i$ containing $e$ in $\cH$. Let $D_e$ be the copy of $C_{2\ell}^{\boxp}$ in $G$ corresponding to $C_i$. Define $X_e \eqdef V(D_e)$. 
		\vspace{-0.75em}
		\item For each $e \in R$ with vertex set $V(e) = \{x, y\}$, define $X_e \eqdef V(e) = \{x, y\}$. (Set theoretically, $V(e)$ and $e$ are the same.) 
	\end{itemize}
	\vspace{-0.5em}
	Through the procedure above, we have associated a set $X_e\subseteq V(G)$ to each vertex $e$ in $V(\cH''') = S \cup R$. 
	
	For every pair of distinct $e_1, e_2 \in V(\cH''')$, we write $e_1 \to e_2$ if and only if $V(e_2)\cap X_{e_1} \ne \varnothing$. By \Cref{fact: min-degree}, we can find a subgraph $\widetilde{\cH}$ of $\cH[S, R]$ with 
	\[
	\delta(\widetilde{\cH}) \ge e(\cH''')/(2m) > e(\cH)/(4m) = \Omega_{\ell}(d). 
	\]
	
	\begin{claim} \label{claim bata good2}
		The binary relation $\to$ is $\widetilde{\beta}$-nice on $\widetilde{\cH}$ for some $\widetilde{\beta} = O_{\ell}(m^{-1/6})$.
	\end{claim}
	
	\begin{poc}
		For any $e_1, e_2 \in V(\widetilde{\cH})$, the definition of $\to$ suggests that $e_2$ has at most $2|X_{e_1}|Td^{1/2}$ neighbors $e_3 \in V(\widetilde{\cH})$ satisfying $e_1 \to e_3$. Indeed, each such neighbor $e_3$ is determined by picking one vertex $a \in X_{e_1}$, one vertex $b$ of $e_2$, and finding a thin $4$-cycle with $a, b$ being diagonal vertices. Similar to \Cref{claim: beta good}, from $2|X_{e_1}|Td^{1/2}/\delta(\widetilde{\cH}) = O_{\ell}(d^{-1/2}) = O_{\ell}(m^{-1/6})$ we see that $\to$ is $O_{\ell}(m^{-1/6})$-nice. 
	\end{poc}
	
	By \Cref{claim bata good2}, $\to$ is $\widetilde{\beta}$-nice on $\widetilde{\cH}$ for some $\widetilde{\beta} = O_{\ell}(m^{-1/6})$. Again, since $\ell \ge 7$ and $m \to \infty$, we obtain
	\[
	\widetilde{\beta} = O_{\ell}(m^{-1/6}) \le \Omega_{\ell} \bigl( m^{-1/\ell} (\log m)^{-4} \bigr) = \Bigl(10^7 \ell^3 v(\widetilde{\cH})^{1/\ell} \bigl( \log v(\widetilde{\cH}) \bigr)^4 \Bigr)^{-1}
	\]
	It then follows from \Cref{lem:goodC2k} that $\widetilde{\cH}$ contains a homomorphic $2\ell$-cycle $C^*=x_1 x_2 \cdots x_{2\ell}$, such that $x_i \to x_j$ fails for all $i \ne j$. Pick an arbitrary $e \in V(C^*) \cap S$ in $V(\widetilde{\cH})$. Then the definitions of $\to$ and $X_e$ show that $C^*$ corresponds to a copy $D^*$ of $C_{2\ell}^{\boxp}$ in $G$. Furthermore, the two copies $D_e, D^*$ of $C_{2\ell}^{\boxp}$ intersect exactly at one edge $e \in E(G) = V(\Gamma)$ and two vertices of $e$. Therefore, the $C_{2\ell}^{\boxp}$ subgraphs $D_e, D^*$ of $G$ combine to a copy of $C^{\Join}_{2\ell}$. The proof of \Cref{thm: 2-de} is complete. 
	
	\subsection{Proof of \texorpdfstring{\Cref{lem:goodC2k}}{Lemma 3.5}}
	
	For graphs $G$ and $H$, denote by $\hom(H, G)$ the number of graph homomorphisms from $H$ to $G$. To establish \Cref{lem:goodC2k}, it suffices to follow the argument of the proof of \cite[Lemma 2.5]{janzer23} verbatim, except that we need to replace the application of \cite[Lemma 2.2]{janzer23} therein by an application of the $\ell = 0$ special case\footnote{To keep the notations consistent with the references \cite{janzer_cycle,janzer23}, we prove \Cref{lem:goodC2k} with $k$ in the place of $\ell$. } in the following asymmetric version of \cite[Lemma 2.4]{janzer_cycle}. 
	
	\begin{lemma} \label{lem: de-cycle}
		Let $k \ge 2$ and $0 \le \ell \le k-1$ be integers and let $G = (V, E)$ be a graph on $n$ vertices. Let $X_1$ and $X_2$ be subsets of $V$. Let $\to$ be a binary relation defined over $V$ such that 
		\vspace{-0.5em}
		\begin{itemize}
			\item for every $u \in V$ and $v \in X_1$, $v$ has at most $\Delta_1$ neighbors $w \in X_{2}$ and amongst them at most $s_1$ satisfy $u \to w$, and
			\vspace{-0.5em}
			\item for every $u \in V$ and $v \in X_2$, $v$ has at most $\Delta_2$ neighbors $w \in X_{1}$ and amongst them at most $s_2$ satisfy $u \to w$. 
		\end{itemize}
		\vspace{-0.5em}
		Let $M = \max\{\Delta_1 s_2, \Delta_2 s_1\}$. Then the number of homomorphic $2k$-cycles 
		\[
		(x_1, x_2, \dots, x_{2k}) \in (X_1 \times X_2 \times X_1 \times \dots \times X_2)\cup (X_2 \times X_1 \times X_2 \times \dots \times X_1)
		\]
		in $G$ such that $x_i \to x_j$ for some $i \ne j$ is at most 
		\[
		64 k^{3/2} M^{1/2} \cdot \hom(C_{2\ell}, G)^{\frac{1}{2(k-\ell)}} \hom(C_{2k}, G)^{1-\frac{1}{2(k-\ell)}}. 
		\]
	\end{lemma}
	
	To derive \Cref{lem: de-cycle}, we need the following graph homomorphism inequality, which is a direct corollary of the log-convexity of $\ell^p$-norms (applied to adjacency matrix spectrum). 
	
	\begin{lemma}[{\cite[Corollary 2.7]{janzer_cycle}}] \label{coro: hom_cycle}
		For any integers $k \ge 2$ and $0 \le \ell \le k-1$ and any graph $G$,
		\[
		\hom(C_{2k-2}, G) \le \hom(C_{2\ell}, G)^{\frac{1}{k-\ell}} \hom(C_{2k}, G)^{1-\frac{1}{k-\ell}}. 
		\]
	\end{lemma}
	
	\begin{proof}[Proof of \Cref{lem: de-cycle}]
		Due to \Cref{coro: hom_cycle}, it suffices to show that the number of homomorphic $2k$-cycles 
		\[
		(x_1, x_2, \dots, x_{2k}) \in \cT \eqdef X_1 \times X_2 \times X_1 \times \dots \times X_2
		\]
		in $G$ with $x_i \rightarrow x_j$ for some $i \ne j$ is upper bounded by $32 k^{3/2} M^{1/2} \cdot \hom(C_{2\ell}, G)^{1/2} \hom(C_{2k}, G)^{1/2}$. By taking into account the $2k$ rotational symmetries and $2$ reflective symmetries of $C_{2k}$, it suffices to prove that the number of homomorphic $2k$-cycles $(x_1, x_2, \dots, x_{2k}) \in \cT$ in $G$ with $x_i \rightarrow x_1$ for some index $i \in \{2, 3, \cdots, k+1\}$ is at most $8 \bigl( kM \cdot \hom(C_{2\ell}, G) \hom(C_{2k}, G) \bigr)^{1/2}$. 
		
		For $a, b \in V(G)$, let $\hom_{a, b}(P_t, G)$ be the number of homomorphic $t$-edge paths $(x_1, x_2, \dots, x_{t+1})$ in $G$ with $x_1 = a, \, x_{t+1} = b$. Denote by $\gamma_{r, t}$ the number of homomorphic $2k$-cycles $(x_1, x_2, \dots, x_{2k}) \in \cT$ with $2^{r-1} \le \hom_{x_1, x_{k+2}}(P_{k-1}, G) < 2^r$ and $2^{t-1} \le \hom_{x_2, x_{k+2}}(P_k, G) < 2^t$ such that there exists some index $i \in \{2, 3, \dots, k+1\}$ satisfying $x_i \to x_1$. Following these notations, we are supposed to prove
		\begin{equation} \label{eq:gamma_bound}
			\sum_{r, t \ge 1} \gamma_{r, t} \le 8\bigl( kM \cdot \hom(C_{2k-2}, G) \hom(C_{2k}, G)\bigr)^{1/2}. 
		\end{equation}
		The proof of \eqref{eq:gamma_bound} is parallel to that of \cite[Lemma 2.5]{janzer_cycle}. Nevertheless, we include it for completeness. 
		
		We first estimate $\gamma_{r, t}$ by counting $(k-1)$-edge paths. Denote by $\alpha_r$ be the number of homomorphic $(k-1)$-edge paths $(y_1, y_2, \dots, y_{k})$ in $G$ with $2^{r-1} \le \hom_{y_1, y_k}(P_{k-1}, G) < 2^r$. It follows that
		\begin{equation} \label{eq:alpha_bound}
			\sum_{r \ge 1} \alpha_r \cdot 2^{r-1} \le \hom(C_{2k-2}, G). 
		\end{equation}
		If $(x_1, x_2, \dots, x_{2k-1}, x_{2k}) \in \cT$ is a homomorphic $2k$-cycles with $2^{r-1} \le \hom_{x_1, x_{k+2}}(P_{k-1}, G) < 2^r$ and $2^{t-1} \le \hom_{x_2,x_{k+2}}(P_k, G) < 2^t$ with $x_i \to x_1$ for some $i \in \{2, 3, \dots, k+1\}$, then 
		\vspace{-0.5em}
		\begin{itemize}
			\item there are at most $\alpha_r$ ways to choose $(x_{k+2}, x_{k+3}, \dots, x_{2k}, x_1)$, and 
			\vspace{-0.5em}
			\item given such a choice, there are at most $\Delta_1$ ways to choose $x_2$, and 
			\vspace{-0.5em}
			\item for each of these choices, there are at most $2^t$ ways to choose $(x_3, x_4, \dots, x_{k+1})$. 
		\end{itemize}
		\vspace{-0.5em}
		So, we deduce that $\gamma_{r, t} \le \alpha_r \Delta_1 \cdot 2^t$ for every $r$ and every $t$. 
		
		We then estimate $\gamma_{r, t}$ by counting $k$-edge paths. Denote by $\beta_t$ be the number of homomorphic $k$-edge paths $(y_1, y_2, \dots, y_{k+1})$ in $G$ with $2^{t-1} \le \hom_{y_1, y_{k+1}}(P_k, G) < 2^t$. It follows that
		\begin{equation} \label{eq:beta_bound}
			\sum_{t \ge 1} \beta_t \cdot 2^{t-1} \le \hom(C_{2k}, G). 
		\end{equation}
		If $(x_1, x_2, \dots, x_{2k-1}, x_{2k}) \in \cT$ is a homomorphic $2k$-cycles with $2^{r-1} \le \hom_{x_1, x_{k+2}}(P_{k-1}, G) < 2^r$ and $2^{t-1} \le \hom_{x_2,x_{k+2}}(P_k, G) < 2^t$ with $x_i \to x_1$ for some $i \in \{2, 3, \dots, k+1\}$, then 
		\vspace{-0.5em}
		\begin{itemize}
			\item there are at most $\beta_t$ ways to choose $(x_2, x_3, \dots, x_{k+2})$, and 
			\vspace{-0.5em}
			\item given such a choice, there are $k$ candidates of $i \in \{2, 3, \cdots, k+1\}$ with $x_i \to x_1$, and 
			\vspace{-0.5em}
			\item for each of these, there are at most $s_2$ choices of $x_1$ which is adjacent to $x_2$ with $x_i \to x_1$, and 
			\vspace{-0.5em}
			\item provided everything above, there are at most $2^r$ choices for $(x_{k+2}, x_{k+3}, \ldots, x_{2k}, x_1)$. 
		\end{itemize}
		\vspace{-0.5em}
		So, we deduce that $\gamma_{r, t} \le \beta_t k s_2 \cdot 2^r$ for every $r$ and every $t$. 
		
		We are ready to establish \eqref{eq:gamma_bound}. Let $q$ be the unique integer with $\lambda \le 2^q < 2\lambda$, where
		\[
		\lambda \eqdef \biggl( \frac{ks_2 \cdot \hom(C_{2k}, G)}{\Delta_1 \cdot \hom(C_{2k-2}, G)} \biggr)^{1/2}. 
		\]
		Dealing with the cases $t < r + q$ and $t \ge r + q$ separately, with the help of \eqref{eq:alpha_bound} and \eqref{eq:beta_bound} we obtain
		\begin{align*}
			\sum_{r, t \ge 1} \gamma_{r, t} &= \sum_{\substack{r, t \ge 1 \\ t < r + q}} \gamma_{r, t} + \sum_{\substack{r, t \ge 1 \\ t \ge r + q}} \gamma_{r, t} = \sum_{\substack{r, t \ge 1 \\ t < r + q}} \alpha_r \Delta_1 \cdot 2^t + \sum_{\substack{r, t \ge 1 \\ t \ge r + q}} \beta_t k s_2 \cdot 2^r \\
			&< \sum_{r \ge 1} \Delta_1 \alpha_r \cdot 2^{r+q} + \sum_{t \ge 1} ks_2\beta_t \cdot 2^{t-q+1} \\
			&< 2\Delta_1 \cdot \hom(C_{2k-2}, G) \cdot 2\lambda + 4ks_2 \cdot \hom(C_{2k, G}) \cdot \lambda^{-1} \\
			&= 8\bigl( kM \cdot \hom(C_{2k-2}, G) \hom(C_{2k}, G)\bigr)^{1/2}. 
		\end{align*}
		This deduces \eqref{eq:gamma_bound}, the proof of \Cref{lem: de-cycle} is complete. 
	\end{proof}
	
	\section*{Acknowledgment}
	
	We thank the two anonymous referees for their careful reading of our manuscript and for their many valuable suggestions, which significantly improved the writing. 
	
	\bibliographystyle{abbrv}
	\bibliography{turan_merge}

\begin{thebibliography}{10}

\bibitem{2023BJS}
D.~Brada{\v c}, O.~Janzer, B.~Sudakov, and I.~Tomon.
\newblock The {T}ur\'an number of the grid.
\newblock {\em Bull. Lond. Math. Soc.}, 55(1):194--204, 2023.

\bibitem{erdos1967}
P.~Erd\H{o}s.
\newblock Some recent results on extremal problems in graph theory. {R}esults.
\newblock In {\em Theory of {G}raphs ({I}nternat. {S}ympos., {R}ome, 1966)},
  pages 117--123 (English); pp. 124--130 (French). Gordon \& Breach, New York,
  1967.

\bibitem{1981ESconj}
P.~Erd\H{o}s.
\newblock Problems and results in graph theory.
\newblock In {\em The theory and applications of graphs ({K}alamazoo, {M}ich.,
  1980)}, pages 331--341. Wiley, New York, 1981.

\bibitem{erdos1983}
P.~Erd\H{o}s.
\newblock Extremal problems in number theory, combinatorics and geometry.
\newblock In {\em Proceedings of the {I}nternational {C}ongress of
  {M}athematicians, {V}ol.\ 1, 2 ({W}arsaw, 1983)}, pages 51--70. PWN, Warsaw,
  1984.

\bibitem{erdos1988}
P.~Erd\H{o}s.
\newblock Some of my old and new combinatorial problems.
\newblock In {\em Paths, flows, and {VLSI}-layout ({B}onn, 1988)}, volume~9 of
  {\em Algorithms Combin.}, pages 35--45. Springer, Berlin, 1990.

\bibitem{erdos1993}
P.~Erd\H{o}s.
\newblock Some of my favorite solved and unsolved problems in graph theory.
\newblock {\em Quaestiones Math.}, 16(3):333--350, 1993.

\bibitem{1966ESS}
P.~Erd\H{o}s and M.~Simonovits.
\newblock A limit theorem in graph theory.
\newblock {\em Studia Sci. Math. Hungar.}, 1:51--57, 1966.

\bibitem{1970Erdos}
P.~Erd\H{o}s and M.~Simonovits.
\newblock Some extremal problems in graph theory.
\newblock In {\em Combinatorial theory and its applications, {I} ({P}roc.
  {C}olloq., {B}alatonf\"{u}red, 1969)}, pages 377--390. North-Holland,
  Amsterdam, 1970.

\bibitem{1946ErdosBAMS}
P.~Erd\H{o}s and A.~H. Stone.
\newblock On the structure of linear graphs.
\newblock {\em Bull. Amer. Math. Soc.}, 52:1087--1091, 1946.

\bibitem{1983ES}
P.~Erd{\H o}s and M.~Simonovits.
\newblock Supersaturated graphs and hypergraphs.
\newblock {\em Combinatorica}, 3(2):181--192, 1983.

\bibitem{frieze_karonski}
A.~Frieze and M.~Karo\'nski.
\newblock {\em Introduction to random graphs}.
\newblock Cambridge University Press, Cambridge, 2016.

\bibitem{2013FSsurvey}
Z.~F\"uredi and M.~Simonovits.
\newblock The history of degenerate (bipartite) extremal graph problems.
\newblock In {\em Erd\"os centennial}, volume~25 of {\em Bolyai Soc. Math.
  Stud.}, pages 169--264. J\'anos Bolyai Math. Soc., Budapest, 2013.

\bibitem{gao_extremal_2023}
J.~Gao, O.~Janzer, H.~Liu, and Z.~Xu.
\newblock Extremal number of graphs from geometric shapes.
\newblock {\em Israel J. of Math.}, to appear.

\bibitem{2022GJNL}
A.~Grzesik, O.~Janzer, and Z.~L. Nagy.
\newblock The {T}ur\'an number of blow-ups of trees.
\newblock {\em J. Combin. Theory Ser. B}, 156:299--309, 2022.

\bibitem{janzer23}
O.~Janzer.
\newblock Disproof of a conjecture of {E}rd{\H{o}}s and {S}imonovits on the
  {T}ur\'an number of graphs with minimum degree 3.
\newblock {\em Int. Math. Res. Not. IMRN}, 10:8478--8494, 2023.

\bibitem{janzer_cycle}
O.~Janzer.
\newblock Rainbow {T}ur\'an number of even cycles, repeated patterns and
  blow-ups of cycles.
\newblock {\em Israel J. Math.}, 253(2):813--840, 2023.

\bibitem{2012SIAMJiang}
T.~Jiang and R.~Seiver.
\newblock Tur\'{a}n numbers of subdivided graphs.
\newblock {\em SIAM J. Discrete Math.}, 26(3):1238--1255, 2012.

\bibitem{1907Mantel}
W.~Mantel.
\newblock Problem 28.
\newblock {\em Wiskundige Opgaven}, 10:60--61, 1907.

\bibitem{1984Simonovits}
M.~Simonovits.
\newblock Extremal graph problems, degenerate extremal problems, and
  supersaturated graphs.
\newblock In {\em Progress in graph theory ({W}aterloo, {O}nt., 1982)}, pages
  419--437. Academic Press, Toronto, ON, 1984.

\bibitem{1941Turan}
P.~Tur\'an.
\newblock Eine {E}xtremalaufgabe aus der {G}raphentheorie.
\newblock {\em Mat. Fiz. Lapok}, 48:436--452, 1941.

\end{thebibliography}

\end{document}